\documentclass[a4paper]{amsart}
\usepackage[american]{babel}

\usepackage[utf8]{inputenc}
\usepackage[unicode]{hyperref}
\usepackage{amsmath,mathtools,amsthm}
\usepackage{graphicx}
\usepackage{newtx}
\usepackage{newtxmath}
\usepackage[mono,narrow,scale=.95]{zi4}
\usepackage[shortlabels]{enumitem}
\usepackage[unicode]{hyperref}
\usepackage{babel}
\usepackage{csquotes}

\usepackage{tikz}
\usetikzlibrary{arrows,automata,matrix,positioning,decorations.pathreplacing}

\usepackage[backend=biber]{biblatex}
\bibliography{document}

\newtheorem{theorem}{Theorem}[section]
\newtheorem{problem}[theorem]{Problem}
\newtheorem{algorithm}[theorem]{Algorithm}

\newtheorem{corollary}[theorem]{Corollary}
\newtheorem{definition}[theorem]{Definition}
\newtheorem{example}[theorem]{Example}

\newtheorem{lemma}[theorem]{Lemma}
\newtheorem{proposition}[theorem]{Proposition}
\newtheorem{remark}[theorem]{Remark}

\newenvironment{algo}[1]{\begin{algorithm}[#1]~\nopagebreak\\
\noindent\begin{tabular}{rl}}{\end{tabular}\hfill\end{algorithm}}

\newcommand{\function}[4]{
\left\{
\begin{array}{@{}r@{~}c@{~}l}
#1 & \longrightarrow & #2 \\
#3 & \longmapsto & #4
\end{array}
\right.
}

\def\O{\mathcal{O}}

\def\R{\mathcal{R}}

\def\FF{\mathbb{F}}

\def\NN{\mathbb{N}}
\def\PP{\mathbb{P}}
\def\QQ{\mathbb{Q}}

\def\ZZ{\mathbb{Z}}
\def\aaa{\mathfrak{a}}
\def\ppp{\mathfrak{p}}

\def\CCC{\mathfrak{C}}

\DeclareMathOperator{\Aut}{Aut}
\DeclareMathOperator{\End}{End}

\DeclareMathOperator{\Hom}{Hom}
\DeclareMathOperator{\Jac}{Jac}

\DeclareMathOperator{\id}{id}

\DeclareMathOperator{\rank}{rk}
\DeclareMathOperator{\spanspan}{span}

\DeclareMathOperator{\tr}{tr}

\DeclareMathOperator{\vol}{vol}

\makeatletter
\g@addto@macro\bfseries{\boldmath}
\g@addto@macro\mdseries{\unboldmath}
\g@addto@macro\normalfont{\unboldmath}
\makeatother

\title[On the computation of endomorphism rings of abelian surfaces]{On the computation of endomorphism rings \\ of abelian surfaces over finite fields}

\author[Anni]{Samuele Anni}
\address{Samuele Anni, Institut de Mathématiques de Marseille, Aix-Marseille Université, France.}

\author[Bisson]{Gaetan Bisson}
\address{Gaetan Bisson, Laboratoire GAATI, University of French Polynesia.}

\author[Iezzi]{Annamaria Iezzi}
\address{Annamaria Iezzi, Univ. Grenoble Alpes, CNRS, Grenoble INP, LJK, 38000 Grenoble, France.}

\author[Lorenzo García]{Elisa Lorenzo García}
\address{Elisa Lorenzo García, Institut de Mathématiques, Université de Neuchâtel, Switzerland,
\newline\phantom{Elisa Lorenzo García, }\hskip\parindent IRMAR UMR 6625, Université de Rennes, France.}

\author[Wesolowski]{Benjamin Wesolowski}
\address{Benjamin Wesolowski, UMPA UMR 5669, ENS de Lyon, France.}

\begin{document}

\begin{abstract}
We study endomorphism rings of principally polarized abelian surfaces over finite fields
from a computational viewpoint with a focus on exhaustiveness.
In particular, we address the cases of non-ordinary and non-simple varieties.
For each possible surface type, we survey known results and, whenever possible,
provide improvements and missing results.
\end{abstract}

\subjclass[2020]{
  	11G07, 11G10, 11G15, 14H52, 14K02, 14K05, 14K10, 14K22.}
\keywords{Abelian surfaces, endomorphism rings, elliptic factors}

\maketitle

\begingroup\def\thefootnote{}
\footnotetext{This work was supported by Agence Nationale de la Recherche under grant ANR-20-CE40-0013.}
\endgroup


\section{Introduction}

Let $A$ be a principally polarized abelian surface defined over a finite field $\FF_q$.
Its endomorphisms defined over the base field form a ring $\End_{\FF_q}(A)$;
the corresponding endomorphism algebra $\QQ\otimes\End_{\FF_q}(A)$ is a division algebra
with center $\QQ(\pi)$, where $\pi$ denotes the Frobenius endomorphism.
Tate \cite{tate} shows that the endomorphism algebra of an abelian variety uniquely identifies its isogeny class.
The ring $\End(A)$ of endomorphisms of $A$ defined over the algebraic closure $\overline\FF_q$,
which we seek to compute, is sometimes strictly larger than $\End_{\FF_q}(A)$ \cite[Theorem~2.4]{silverberg-fields-of-def}.
This ring is an order in the division algebra $K=\QQ\otimes\End(A)$ stable under complex conjugation and containing $\ZZ[\pi,\overline\pi]$, where $\overline\pi=q/\pi$.

Endomorphism rings of abelian varieties are finer-grained invariants than endomorphism algebras. Their computation allows one to efficiently partition isogeny classes into smaller components, making them highly relevant to both computational number theory and cryptography, with  numerous applications, including the evaluation of the hardness of the discrete logarithm problem \cite{jao-miller-venkatesan-dlog,brooks-jetchev-wesolowski} and the computation of class and modular polynomials \cite{drew-hilbert,drew-modpol,drew-modpol-eval}.

We consider two effective variants of this problem.

\begin{problem}\label{prob:abstract}
Given a principally polarized abelian surface defined over a finite field,
obtain an abstract representation of its endomorphism ring, that is,
a canonical division algebra (such as
$\QQ[x]/(x^4+Ax^2+B)$ in the simple, ordinary case) isomorphic to the endomorphism algebra together with an explicit subring isomorphic to the endomorphism ring.
\end{problem}

\begin{problem}\label{prob:explicit}
Given a principally polarized abelian surface defined over a finite field,
obtain an explicit generating set of endomorphisms which can be efficiently evaluated.
\end{problem}

Both problems can be related through an explicit embedding $K\to\QQ\otimes\End(A)$
where $K$ is a division algebra which depends only on the isogeny class \cite{tate}.
Here, our efforts focus on Problem~\ref{prob:explicit}.

Let $p$ denote the characteristic of the base field, such that we have
$q=p^n$ for some positive integer $n$. For an abelian variety of dimension $g$, the $p$-torsion of $A$
satisfies $A[p]\simeq(\ZZ/p\ZZ)^r$ where the integer $0\leq r\leq g$ is called
the \emph{$p$-rank} of $A$ and is denoted by $r(A)$. The $p$-rank is invariant under isogenies and
satisfies $r(A\times B)=r(A)+r(B)$ for every pair of abelian varieties $A$ and $B$ over $\FF_q$.
Abelian varieties of $p$-rank $g$ are called \emph{ordinary} and form the generic case: the moduli space
of ordinary abelian varieties of dimension $g$ has dimension $g(g+1)/2$, which is also the dimension of
the entire moduli space of abelian varieties of dimension $g$. In contrast, abelian varieties whose
$p$-rank vanishes are said to be \emph{supersingular}.

Abelian varieties of dimension $g=1$ are known as elliptic curves and are either ordinary or supersingular.
Their endomorphism algebra is an imaginary quadratic field in the ordinary case and a quaternion algebra in the supersingular case.
Computing their endomorphism rings was first addressed by Kohel \cite{kohel-phd} who provided explicit algorithms of exponential complexity.
In the ordinary case, this was improved to subexponential-time algorithms by Bisson and Sutherland \cite{endomorphism,grh-only} and, more recently, to a polynomial-time algorithm by Robert \cite{Robert22c}.
In the supersingular case, state-of-the-art algorithms remain of exponential complexity \cite{EHLMP20,FIKMN23,PW23}. 

Abelian varieties of dimension $g=2$ are known as abelian surfaces and their $p$-rank is either $0$, $1$, or $2$.
Abelian surfaces of $p$-rank $2$ form a strata of dimension $3$ of the moduli space; those of $p$-rank $1$ and $0$ form a strata of respective dimension $2$ and $1$, see \cite[Theorem 2.3]{GP05}.
Methods for computing their endomorphism rings were designed only in the ordinary and absolutely simple case, first by Eisenträger and Lauter \cite{eisentrager-lauter} who described an algorithm of exponential complexity in $\log(q)$ and later by Bisson \cite{end-g2} who obtained a subexponential algorithm.
\medskip

\noindent\textbf{Our contribution.}\quad
We classify abelian surfaces according to their $p$-rank and whether they are
absolutely simple. Table~\ref{table:cases} enumerates all cases, each of which will be
the topic of a specific section.

\begin{table}
\newcommand{\mtlines}[1]{\begin{tabular}[c]{@{}c@{}}#1\end{tabular}}
\begin{center}
\begin{tabular}{|c||c|c|c|}
\multicolumn{4}{c}{\textsc{simple abelian surfaces}}\\
\hline
$p$-rank & surface type & $\End$ type & Section \\
\hline
\hline
2 & $\Jac(C)$ & $\O$ & \ref{sec:simple-ordinary} \\
\hline
1 & $\Jac(C)$ & $\O$ & \ref{sec:simple-prank1} \\
\hline
\multicolumn{4}{c}{~} \\
\multicolumn{4}{c}{\textsc{non-simple abelian surfaces}} \\
\hline
$p$-rank & surface type & $\End$ type & Section \\
\hline
\hline
2
& \mtlines{$E_1\times E_2$ \\ $E_1$ and $E_2$ ordinary}
& $\begin{pmatrix} \O_1 & \aaa \\ \widehat\aaa & \O_2 \end{pmatrix}$
& \ref{sec:non-simple-prank2} \\
\hline
1
& \mtlines{$(E_1\times E_2)/H$ \\ $E_1$ ordinary, $E_2$ supersingular, $H$ finite}
& suborder of $\O_1\times\R_2$
& \ref{sec:non-simple-prank1} \\
\hline
0
& \mtlines{$E^2$ or $E^2/\alpha$ \\ with $\alpha$ an $\alpha$-group}
& $\begin{pmatrix} \R & \R \\ \R & \R \end{pmatrix}$
& \ref{sec:non-simple-prank0} \\
\hline
\end{tabular}
\bigskip
\end{center}
\caption{Types of abelian surfaces and associated endomorphism rings,
where $\O$ denotes an order in a CM-field,
$\R$ a maximal order in the quaternion algebra $\QQ_{p,\infty}$,
and $\aaa$ is an ideal. See the relevant sections for details.}
\label{table:cases}
\end{table}

The nature of the results we obtain varies with the type of surfaces.
This article's first objective is to survey the literature and, whenever possible,
to improve upon the state-of-the-art or to fill in missing details.
In particular we prove the existence of a general algorithm to solve problem \ref{prob:explicit}.
In the case of simple surfaces of $p$-rank $1$ we show that known methods can be applied
and, in the case of non-simple surfaces, we presents new algorithms to compute the associated elliptic factors.

Our algorithms vary in efficiency. Nevertheless their existence holds intrinsic value, as they contribute to a deeper understanding of the landscape and open avenues for further refinement and exploration.

Section~\ref{sec:preliminaries} contains preliminaries on representing surfaces and computing basic invariants.
Section~\ref{sec:computability} answers the problem of theoretical computability of endomorphism rings and, in particular, of generic endomorphism testing. 
Section~\ref{sec:simple} tackles the case of simple surfaces and discusses the lattice of orders, a classical challenge towards problem \ref{prob:explicit}. 
Section~\ref{sec:non-simple} deals with non-simple surfaces, where the situation demands a more nuanced approach; in particular, we describe two distinct algorithms to find elliptic subcovers.
Section~\ref{sec:extra-auto} finally considers the particular case of surfaces with extra automorphisms, where elliptic factors can be explicitly given.


\section{Preliminaries}
\label{sec:preliminaries}


\subsection{Representing abelian surfaces and their endomorphisms}

Unless otherwise specified, all abelian surfaces are implicitly assumed
to be defined over a finite field and endowed with a principal polarization.
We represent them differently depending on their type as per the following theorem.

\def\citation{\cite[Satz 2]{weil-torelli}}
\begin{proposition}[\citation]
Every principally polarized abelian surface is either
the Jacobian variety of a genus-two curve or
the product of two elliptic curves with the product polarization.
\end{proposition}

For simple surfaces, we rely on the representation given by the celebrated theorem below.

\begin{theorem}[\cite{weil-torelli} and {\cite[Section 5.10]{Adleman}}]
Let $A$ be a principally polarized abelian surface defined over a finite field $\FF_q$.
If $A$ is simple over the quadratic extension $\FF_{q^2}$, then
$A$ is $\FF_q$-isomorphic to the Jacobian of a projective smooth curve of genus two.
\end{theorem}

Non-simple surfaces are either:
\begin{itemize}
\item isomorphic to a product of isogenous
elliptic curves; 
\item Jacobian varieties of algebraic curves which are not isomorphic to such a product.
\end{itemize}

The endomorphism ring of a product of two elliptic curves $E_1$ and $E_2$ satisfies
\[
\End(E_1\times E_2)=\begin{pmatrix}\End(E_1) & \Hom(E_2,E_1) \\ \Hom(E_1,E_2) & \End(E_2)\end{pmatrix}
\]
and its computation is thus reduced to computing
their individual endomorphism rings as well as an isogeny between them
which may be obtained using the methods of \cite{galbraith-isogenies}.
For surfaces $A$ not isomorphic to such a product,
the first step is to identify two elliptic curves $E_1$, $E_2$ and an isogeny $A\to E_1\times E_2$.
This is addressed in Section~\ref{sec:ellfactor}.

Henceforth, we thus focus on surfaces given as the Jacobian variety of a genus-two curve
for which points can be represented in Mumford coordinates \cite{mumford-tata2} and
the group law computed using Cantor's algorithm \cite{cantor}.
This representation can be extended to non-simple 
abelian surfaces which belong to the isogeny class of a Jacobian variety:
such surfaces $A$ can be represented as a couple $(\varphi,C)$ whereby $\varphi:\Jac(C)\to A$ is an isogeny.
Note however that this representation excludes some non-simple abelian surfaces \cite[Theorem~1]{howe-maisner-nart-ritzenthaler}.

Separable isogenies of simple abelian surfaces and, in particular, their endomorphisms,
may be represented by their kernel since they are finite;
from a kernel, the corresponding isogeny can be efficiently evaluated using Vélu's formulas \cite{velu}
and later improvements \cite{bernstein-defeo-leroux-smith}.
Isogenies of non-simple abelian surfaces are represented more naïvely as algebraic maps given by tuples of rational fractions.


\subsection{Computing basic invariants}

Let $A=\Jac(C)$ be the Jacobian variety of a genus-$2$ hyperelliptic curve $C$
defined over a finite field $\FF_q$ where $q=p^n$. Let
\[
f_{A}(t)=t^4+a_1t^3+a_2t^2+qa_1t+q^2\in \mathbb Z[t]
\]
denote the characteristic polynomial of its Frobenius endomorphism $\pi$ and set 
\[
\Delta:=a_1^2-4a_2+8q,\qquad \delta:=(a_2+2q)^2-4qa_1^2.
\]

We can determine whether $A$ is absolutely simple
from the coefficients of $f_A$ using \cite[Theorem~6]{howe-zhu}.
Moreover, the variety $A$ has $p$-rank two if and only if
$p\nmid a_2$ and $\Delta$ is not a square in $\ZZ$;
it has $p$-rank one if and only if $p\nmid a_1$,  $v_p(a_2)\geq \frac{n}{2}$, $\delta$ is not a square in $\mathbb Z_p$ and $\Delta$ is not a square in $\ZZ$;
see \cite[Theorem~2.9]{maisner-nart}.

We note that, alternatively, the $p$-rank can be determined by
looking at the splitting pattern of $p$ in the CM-field $K$
and by the method of \cite{oconnor-mcguire-naehrig-streng}.

We will sometimes also use the $a$-number which is defined as $a(A)=\dim\Hom_{\FF_p}(\alpha_p,A)$
where $\alpha_p$ is the only local-local group scheme over $\FF_p$.
See \cite[Lemma 2.2]{EP07} for its explicit computation.

Henceforth, we denote by $A\sim B$ the fact that the abelian varieties $A$ and $B$ are isogenous.
We denote by $R\simeq R'$ the fact that the groups or algebras $R$ and $R'$ are isomorphic.


\section{Computability of endomorphism rings}
\label{sec:computability}

Here we present two generic methods for computing endomorphism rings of abelian surfaces.
They are of theoretical interest merely because they prove their computability in the general case.
However their time complexity are prohibitive and subsequent sections will present specialized methods which achieve better complexities for each subcase.


\subsection{Generic endomorphism testing}

We borrow the following definition from~\cite{Wes21}
and refer the reader to~\cite{benhdr} for a more precise statement.

\begin{definition}\label{def:efficient-representation}
Let~$\varphi : A \to B$ be an isogeny between two abelian varieties defined over a finite field $\FF_q$.
An \emph{efficient representation} of $\varphi$ with respect to a given algorithm is some data $D_\varphi \in \{0,1\}^*$ such that, on input $D_\varphi$ and $P \in A(\FF_q)$, the algorithm returns the evaluation $\varphi(P)$ in polynomial time in $\operatorname{length}(D_\varphi)$ and $\log q$.
\end{definition}

We denote by $\alpha^\dagger$ the Rosati involution of an endomorphism $\alpha$ of an abelian variety $A$.
This yields a positive definite bilinear form on $\End(A)$ defined by $\langle\alpha,\beta\rangle=\tr(\alpha\circ\beta^\dagger)$.
The quadratic structure is computationally available thanks to the following lemma.

\begin{lemma}\label{lemma:scalarproductofisogenies}
Let $A$ be a principally polarized abelian surface. Given
two endomorphisms $\alpha,\beta \in \End(A)$, an efficient representation of $\alpha$ and $\beta^\dagger$, and an integer $D$ such that $\langle \alpha,\alpha\rangle,\langle \beta,\beta\rangle < D$, one can compute
$$\langle \alpha,\beta\rangle = \tr(\alpha \circ \beta^\dagger)$$
in polynomial time in the length of the input.
\end{lemma}

\begin{proof}
This is inspired by~\cite[Lemma~7]{PL17}, which itself follows a strategy similar to Schoof’s point counting algorithm.
The trace of an endomorphism is the trace of its action on the Tate module $T_\ell(A)$ for any prime $\ell$. Thus the action of $\alpha \circ \beta^\dagger$ on $A[\ell]$ reveals $\tr(\alpha \circ \beta^\dagger) \bmod \ell$. Since $|\tr(\alpha \circ \beta^\dagger)| \leq D$, it is sufficient to evaluate the action of $\alpha \circ \beta^\dagger$ on $A[\ell]$ for small primes $\ell$ such that $\prod_\ell \ell > 2D$, then recover $\langle \alpha,\beta\rangle$ with the Chinese remainder theorem.
\end{proof}

\begin{definition}
A \emph{good representation} for an isogeny $\varphi : A \to B$  is a triple $(r,r^\dagger,D)$ where $r$ is an efficient representation of $\varphi$, $r^\dagger$ is an efficient representation of $\varphi^\dagger = \lambda_A^{-1} \circ \varphi^\vee \circ \lambda_B$, where $\lambda_A$ and $\lambda_B$ denote the respective polarizations of $A$ and $B$, and $D$ is an integer such that $\tr(\varphi \circ \varphi^\dagger) \leq D$.
\end{definition}

\begin{remark}
Let $\varphi$ be an $(\ell,\ell)$-isogeny.
Its kernel $K\subset A[\ell]$ is an efficient representation of $\varphi$.
Furthermore, the $(\ell,\ell)$-isogeny $\varphi^\dagger$ has kernel $\varphi(A[\ell])$.
We deduce that $(K,\varphi(A[\ell] ),\ell)$ is a good representation of $\varphi$.
\end{remark}

Per Lemma~\ref{lemma:scalarproductofisogenies}, given a good representation for $\alpha,\beta \in \End(A)$, one can compute $\langle \alpha,\beta\rangle$ in polynomial time in the length of the input.

\begin{proposition}\label{prop:augment-subring}
There exists an algorithm which, given a collection of endomorphisms $\alpha = (\alpha_i)_{i=1}^n \in \End(A)$ in good representation, outputs a good representation of a basis of $\spanspan_{\mathbb Q}(\alpha) \cap \End(A)$. In particular, if $\alpha$ has full rank, the output is a basis of $\End(A)$.
\end{proposition}

\begin{proof}
From Lemma~\ref{lemma:scalarproductofisogenies}, one can compute the Gram matrix $G = (\langle \alpha_i, \alpha_j\rangle)_{i,j}$ for any collection $\alpha$. Up to selecting a subset, we can assume the $\alpha_i$ linearly independent, hence $G$ is of dimension and rank $n$. Let $S = \spanspan_{\mathbb Z}(\alpha)$ and $R = \spanspan_{\mathbb Q}(\alpha) \cap \End(A)$. We have $S \subset R$, and $[R:S] = \vol(S)/\vol(R)$ (where the volume is with respect to the scalar product $\langle -,-\rangle$). We have $\vol(S)^2 = \det(G) \in \mathbb Z$, and similarly, $\vol(R) \in \mathbb Z$ since $\langle -,-\rangle$ is integral on $\End(A)$.
In particular, $[R:S]^2$ is a divisor of $\det(G)$.
The algorithm then proceeds as follows. For each prime $\ell$ whose square divides $\det(G)$:
\begin{enumerate}[(Step 1)]
    \item \label{augment-step-1} Let $L \subset S$ be a list of representatives of the finite quotient $S/\ell S$.
    \item \label{augment-step-2} For each $\beta \in L$, if $\beta(A[\ell]) = 0$, then we have $\beta/\ell \in R$, we find a basis of $S + \mathbb Z \cdot \beta$, and update $S$ and $G$ accordingly to correspond to this larger ring. Return to~\ref{augment-step-1} with the same prime $\ell$.
    \item \label{augment-step-3} If no $\beta$ with $\beta(A[\ell]) = 0$ was found we have $\ell \nmid [R:S]$ (i.e., we have reached maximality locally at $\ell$), return to~\ref{augment-step-1} with the next prime $\ell$.
\end{enumerate}
Termination follows from the fact that upon each return to~\ref{augment-step-1}, either $[R:S]$ has been divided by a factor $\ell$ (which can only happen finitely many times before reaching $[R:S] = 1$), or one progresses forward in the list of prime factors of $\det(G)$. Correctness follows from the fact that for each $\ell$, one eventually reaches~\ref{augment-step-3}, at which point $[R:S]$ is guaranteed not be be divisible by $\ell$ anymore; so once the list of prime factors of $\det(G)$ is exhausted, we get that $[R:S]$ has no prime factor, hence $[R:S] = 1$.
\end{proof}

Note that, the list $L$ computed in \ref{augment-step-1} is of exponential length in $\log(\ell)$ and hence in the input size.
Testing whether $\beta(A[\ell]) = 0$ in \ref{augment-step-2}, when done naively, is also of exponential complexity.

By Proposition~\ref{prop:augment-subring}, it only remains to prove that there exists an algorithm that produces a full-rank collection of endomorphisms of $A$. One can compute the rank of any collection as the rank of the Gram matrix. Since the rank of $\End(A)\otimes \mathbb Q$ is known, there is an algorithm to check whether a collection has full rank.

\begin{theorem}
There exists an explicit algorithm that, given a principally polarized abelian surface $A$ defined over a finite field $\FF_q$, outputs a basis of its endomorphism ring $\End_{\FF_q}(A)$.
\end{theorem}

\begin{proof}
At this point, we are only concerned with showing that $\End_{\FF_q}(A)$ is computable, with no concern for efficiency. We therefore propose the following naïve strategy. Given a principally polarized abelian surface $A$, exhaustively enumerate all endomorphisms of $A$ by enumerating all possible maps (as tuples of rational fractions of increasing degrees) and testing which are indeed endomorphisms (e.g., testing that the maps indeed send $A$ to itself using Gröbner bases).
This results in a sequence $(\alpha_i)_{i=0}^n$ such that $\End_{\FF_q}(A) = \{\alpha_i\}_i$, and the algorithm generates each $\alpha_i$ in that order. 

From any initial sequence $(\alpha_i)_{i=0}^n$, one can compute the corresponding Gram matrix, and a basis of the lattice it generates. The main difficulty is in deciding at which $n$ to stop (i.e., when $\spanspan_{\mathbb Z}(\alpha_i)_{i=0}^n = \End_{\FF_q}(A)$). This is where Proposition~\ref{prop:augment-subring} comes in: it is sufficient to reach a point where $\rank(\spanspan_{\mathbb Z} (\alpha_i)_{i=0}^n) = \rank(\spanspan_{\mathbb Z} (\End_{\FF_q}(A)))$, and the algorithm of Proposition~\ref{prop:augment-subring} takes care of the rest. The quantity $\rank(\spanspan_{\mathbb Z} (\alpha_i)_{i=0}^n)$ is the rank of the Gram matrix, computable by Lemma~\ref{lemma:scalarproductofisogenies}. The quantity $\rank(\spanspan_{\mathbb Z} (\End_{\FF_q}(A)))$ is the rank of the endomorphism algebra, which can be computed in polynomial time.
\end{proof}

Note that all endomorphisms are defined over extensions of the base field of bounded degree \cite[Theorem~2.4]{silverberg-fields-of-def}.
Therefore, to compute the ring $\End(A)$ of endomorphisms defined over the algebraic closure,
one only needs to run the above algorithm over a bounded extension of the base field.


\subsection{Lifting to characteristic zero}

Several algorithms have been designed for the computation of endomorphism rings of abelian varieties defined over a field of characteristic zero \cite{bruin-sijsling-zotine,EndMany,EndLombardo}
and have been used to verify the correctness of the endomorphism data in the $L$-functions and modular forms database (LMFDB) \cite{LMFDB}
which contains $66,158$ curves of genus two with small minimal absolute discriminant as of February 2025.

Since abelian surfaces of positive $p$-rank defined over finite fields admit a canonical lift \cite{LST}, the computation of the abstract structure of their endomorphism rings may be transported to characteristic zero.
Nevertheless, the computation of such canonical lifts is exponential in $\log(p)$ \cite{maiga-robert-odd,maiga-robert-even,maiga-robert-medium}; furthermore, it is unclear whether characteristic-zero methods for computing endomorphism rings yield better overall complexity than their finite fields counterpart as we are unaware of rigorous complexity bounds for those methods.

Note that, one could avoid the computation of the canonical lift, by taking any lift of the curve with extra endomorphisms \cite{oort-lifting} and thus obtain the right order up to an index a power of $p$ (see Corollary 6.1.2. in \cite{GL12}).

\begin{lemma}
Let $A$ be an abelian variety defined over a number field $K$ and $\mathfrak{p}$ be a prime of good reduction of norm $p$. The reduction map $\iota:\End{A}\rightarrow\End{A_{\mathfrak{p}}}$ is injective and $[(\iota(\End{A})\otimes\QQ)\cap\End{A_{\mathfrak{p}}}:\End{A}]$ is a power of $p$.
\end{lemma}

\begin{example}
Consider the hyperelliptic curve defined over the field $\FF_{11}$ by
\[
C:y^2=x^6 + 6x^5 + 4x^4 + 2x^3 + 5x^2 + 7x + 2.
\]
Its Frobenius endomorphism $\pi$ has characteristic polynomial $t^4+10t^2+121$
and we find that the order $\ZZ[\pi,\overline\pi]$ has index $2^5$ in the ring of integers of the quartic field $K=\QQ(\pi)$.

A lift of this curve which admits extra endomorphisms is
\[
C':y^2 + x^3y = x^3 + 2;
\]
which is referenced by LMFDB as ``Genus 2 curve \texttt{5184.a.46656.1}''.
The endomorphism ring of its Jacobian variety is $\End_{\QQ}(\Jac(C'))=\ZZ[\sqrt{-2}]$;
thus, the endomorphism ring of $\Jac(C)$ contains the order $\O=\ZZ[\pi,\overline\pi,\sqrt{-2}]$
which has index $2$ in the maximal order $\O_K$.
This reduces the possibilities for $\End(\Jac(C))$ to just two cases: $\O$ and $\O_K$.
By computing $(2,2)$-isogenies, we eliminate the case $\O_K$ and deduce that $\End(\Jac(C))=\O$.
\end{example}


\section{Simple abelian surfaces}
\label{sec:simple}

Simple abelian surfaces are either of $p$-rank 2 (ordinary) or of $p$-rank 1.


\subsection{Simple, ordinary case}
\label{sec:simple-ordinary}

When $A$ is ordinary, its endomorphism algebra $K=\QQ(\pi)$ is a quartic CM-field, that is, an imaginary quadratic extension of a totally real number field $K^+$.
The endomorphism ring $\End(A)$ is an order of $K$ containing $\ZZ[\pi,\overline\pi]$ and stable under complex conjugation. Conversely,
all such orders are endomorphism rings, see  \cite{waterhouse}.
In particular, the conductor of $\End(A)$ divides the index $\nu=\left[\O_K:\ZZ[\pi,\overline\pi]\right]$ where $\O_K$ denotes the ring of integers of $K$.

Classically, the problem of computing $\End(A)$ has been split into two subproblems:
first, to determine whether a given order $\O$ is contained in the endomorphism ring $\End(A)$;
second, to select suitable candidate orders $\O$ so as to determine $\End(A)$.
The latter is covered in Section~\ref{sec:ascending}; from a high level perspective, it computes
the localization of $\End(A)$ locally at each prime $\ell$ dividing the conductor $\nu$,
from which $\End(A)$ is eventually deduced.

\subsubsection{Testing candidate orders}

Assume a candidate order $\O$ of $K$ is fixed and we wish to determine whether
$\O\subset\End(A)$ holds. Depending on the prime $\ell$ and other factors, 
we might elect to choose one of the methods below.
\smallskip

\paragraph{\bfseries The method of Eisenträger--Lauter \cite{eisentrager-lauter}.}
This method exploits the following observation: let $\alpha$ be an endomorphism of $A$ and let $\ell$ be an integer coprime to $p$;
then we have $\alpha/\ell\in\End(A)$ if and only if $A[\ell]\subset\ker(\alpha)$ holds.
This can be computed efficiently when the $\ell$-torsion subgroup $A[\ell]$ is defined over relatively small extension fields.
Generically, however, the torsion $A[\ell]$ may be defined over extensions of degree as large as $\ell^g$
and the primes $\ell$ themselves, being factors of $\nu$, are only bounded by $2^{g(2g-1)}q^{g^2/2}$.
This algorithm thus has an exponential complexity in the worst case.
\smallskip

\paragraph{\bfseries The method of Bisson \cite{end-g2}.}
This approach exploits complex multiplication theory and, more specifically, the faithful action of the polarized class group $\CCC(\O)$ of Shimura \cite{shimura-taniyama}
on the set of isomorphism classes of principally polarized abelian varieties with endomorphism ring $\O$.
The main idea is to construct polarized ideals which are trivial in $\CCC(\O)$, to evaluate the corresponding isogenies, and to check that those are in fact endomorphisms. When we have $\O\subset\End(A)$, this is always the case.
Conversely, if all principal polarized ideals map to endomorphisms, one shows \cite{quartic-ccc} that the order $\O$ is never ``far'' from being a suborder of $\End(A)$.
Applying the method of Eisenträger--Lauter locally at small primes allows one to overcome this and obtain a full converse.
This yields an algorithm of heuristic, subexponential complexity to determine whether the inclusion $\O\subset\End(A)$ holds.

An abelian variety $A$ is said to have \emph{maximal real multiplication}
if its endomorphism ring $\End(A)$ contains $\O_{K^+}$. In this particular case,
assuming furthermore that the narrow class group of $K^+$ is trivial,
Springer \cite{springer} improves Bisson's method by exploiting the classical class group instead of the polarized class group. This yields an algorithm which relies on fewer heuristics.
\smallskip

\paragraph{\bfseries The method of Robert \cite{Robert22c}.}
Robert exploits Kani's diamond lemma \cite{Kani97} to obtain a polynomial-time algorithm testing whether certain maps of elliptic curve are in fact endomorphisms.
Although it is written in the case of ordinary elliptic curves, there is no obstacle to its generalization to simple, ordinary abelian surfaces.
However, this does not necessarily give a polynomial-time algorithm for computing endomorphism rings, since this depends on the number of orders to test, a problem to which we now turn.

\subsubsection{Ascending the lattice of orders}
\label{sec:ascending}

The general idea of \cite[Algorithm 6.2]{end-g2} is as follows:
starting from $\O'=\ZZ[\pi,\overline\pi]$, iterate over orders $\O$ which are \emph{directly above} $O'$
(that is, such that $\O' \subset \O$ and no other order strictly lies between them);
if $\O\subset \End(A)$, then set $\O'\leftarrow O$ and repeat until there are no more orders to test;
eventually, return $\End(A)=\O'$. 

For elliptic curves, the CM-field $K$ is a quadratic CM-field.
Locally at any prime $\ell$, its lattice of orders is thus linear and there is only one order to consider above every given one;
this can be efficiently exploited, see \cite[Section 5]{grh-only}.
Coupled with Robert's testing method, this yields a polynomial-time algorithm for computing the endomorphism ring of ordinary elliptic curves.

In quartic CM-fields, however, there can be exponentially many such orders stable under complex conjugation.
In particular, this is always the case when $\ell^3\mid [\O_K:\O]$; see \cite[Lemma 5.3]{fieker-hofmann-sanon}.
One may try to limit the number of orders to test by first computing the real part of the endomorphism ring.

Denote by $K^+$ the maximal totally real subfield of $K$, let $\O=\End(A)$ and let $\O_+=\O\cap \O_{K^+}$.
Note that $\O_+$ can be computed in polynomial time, as $K^+$ is quadratic and therefore its lattice of orders is linear locally at each prime $\ell$.
Let $\O^\#$ be the largest order in $\O_K$  such that $\O^\#\cap\O_{K^+}=\O_+$, which exists as a consequence of the following lemma.

\begin{lemma}\label{lemmaO+}
For $i=1,2$, let $\O_i$  be two suborders of $\O_K$ stable under complex conjugation such that $\O_i\cap\O_{K^+}=\O_+$. Locally at any prime $\ell\neq 2$ we have $(\O_1+\O_2)\cap\O_{K^+}=\O_+$.
\end{lemma}

\begin{proof}
Let $x\in (\O_1+\O_2)\cap\O_{K^+}$. We have $x=x_1+x_2\in\mathbb R$ with $x_i\in\O_i$. Thus $2x=x+\bar{x}=x_1+\bar{x}_1+x_2+\bar{x}_2\in (\O_1\cap\O_{K^+})+(\O_2\cap\O_{K^+})=O_+$.
\end{proof}

We thus have the inclusions of orders displayed in Figure \ref{fig:realorders}.

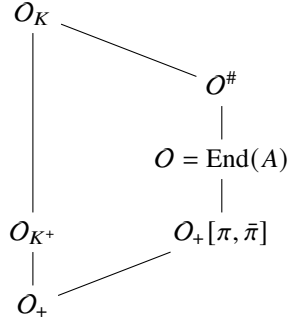
\begin{figure}
\begin{center}
\begin{tikzpicture}[description/.style={fill=white,inner sep=pt}]
\matrix(m) [matrix of math nodes,  row sep=1.3em,
column sep=3em, text height=1.5ex, text depth=0.25ex, ampersand replacement=\&]
{  \O_K  \&   \\
   \& \O^\#\\
   \& \O=\End(A)\\
 \O_{K^+}    \& \O_+[\pi,\bar{\pi}] \\
     \O_+ \& \\};
\path[-]
(m-1-1) edge  (m-2-2)
(m-1-1) edge  (m-4-1)
(m-2-2) edge  (m-3-2)
(m-3-2) edge  (m-4-2)
(m-4-1) edge  (m-5-1)
(m-4-2) edge  (m-5-1)
;
\end{tikzpicture}
\end{center}
\caption{Inclusions between orders in a quartic CM-field and their intersections with the totally real subfield.}
\label{fig:realorders}
\end{figure}

However, the result below, which is a generalization of \cite[Lemma~5.3]{fieker-hofmann-sanon},
shows that there are still exponentially many orders within the resulting bounds for $\End(A)$.

\begin{lemma}
Let $K$ be a quartic CM-field; denote by $K^+$ its totally real subfield.
Let $\O$ be an order in $K$ stable under complex conjugation $\sigma$.
Let $\ell$ be an odd prime.
The number of orders which are stable under complex conjugation, contained in $\O\cap K^++\ell\O$ and containing $\O\cap K^++\ell^2\O$, is greater than $\ell$.
\end{lemma}

\begin{proof}
These orders correspond to submodules over $\FF_\ell[\{1,\sigma\}]$ of $(\O\cap K^++\ell\O)/(\O\cap K^++\ell^2\O)$.
By Maschke's theorem, $\FF_\ell[\{1,\sigma\}]$ is semisimple: it admits two absolutely irreducible modules, $V_+$ and $V_-$,
each of dimension one over $\FF_\ell$. Thus, the quotient above is isomorphic to $V_+^{n_+}\oplus V_-^{n_-}$ where $n_+$ and $n_-$
are non-negative integers whose sum equals the dimension of the quotient, that is, two.
Since the action of $\sigma$ on $V_-$ is nontrivial and it stabilizes the direct sum, the integer $n_-$ is even.
Thus, one of $n_+$ or $n_-$ equals two and the corresponding module $V^n$ has $\frac{\ell^2-1}{\ell-1}=\ell+1$ submodules.
\end{proof}

We conclude that classical methods for computing the endomorphism rings by ascending the lattice of orders
cannot have subexponential worst-case complexity in the general case.


\subsection{Simple, $p$-rank-$1$ case}
\label{sec:simple-prank1}
Recall that $p$-rank-1 abelian surfaces may be efficiently detected via the following result.

\begin{lemma}[{\cite[Lemma 1]{oconnor-mcguire-naehrig-streng}}]
    A simple abelian surface $A$ defined over $\mathbb F_q$, with $q=p^n$, has $p$-rank $1$ if and only if the following conditions are satisfied:
    \begin{enumerate}
        \item the field $K=\mathbb Q(\pi)$ is a quartic CM-field,
        \item\label{lem:cond2} the prime $p$ splits in $K$ as $p\O_K=\ppp_1\overline{\ppp_1}\ppp_2^e$, where $e\in\{1,2\}$, and
        \item we have $\pi\O_K=\ppp_1^n\ppp_2^{en/2}$, with $e$ as in Condition (\ref{lem:cond2}).
    \end{enumerate}
\end{lemma}

In particular, endomorphism rings of $p$-rank-$1$ surfaces are
specific orders among those of ordinary surfaces.
The techniques of Section~\ref{sec:ascending} therefore apply equally
to the case of $p$-rank-$1$ surfaces.

To compute the endomorphism rings, these techniques may be coupled with
the method of Eisenträger and Lauter \cite{eisentrager-lauter} or even of Bisson \cite{end-g2}.
Indeed, by the theory of Shimura and Taniyama \cite{shimura-taniyama},
isogenies between $p$-rank-$1$ surfaces correspond to ideals of orders of the endomorphism algebra.
Since this algebra is of the same type as those of ordinary surfaces,
the free action of \cite[Section~2]{end-g2},
the endomorphism ring testing routine of \cite[Section~5]{end-g2},
the resulting method \cite[Algorithm~6.2]{end-g2},
the main result \cite[Theorem~7.1]{end-g2} and its proof
all apply without and modification to abelian surfaces of $p$-rank-$1$.

\begin{example}
Consider the Jacobian variety $A$ of the genus-two curve
\[
y^2 = 23535 x^6 + 6448 x^5 + 20387 x^4 + 3811 x^3 + 11376 x^2 + 11282 x + 21340
\]
defined over the finite field with $36877$ elements.
It is absolutely simple, has $p$-rank~1, and satisfies $[\O_K:\ZZ[\pi,\overline\pi]]=431$ with the $431$-torsion of $A$ being defined over an extension of degree $5003910$.
In the ring of integers of $K$, we have $3=\aaa\overline\aaa$ where $\aaa$ is a prime ideal of norm $9$.
The element $(\aaa,3)$ is of order $736$ in the polarized class group $\CCC(\O_K)$ but not in $\CCC(\ZZ[\pi,\overline\pi])$.
We compute the sequence of $736$ isogenies of type $(3,3)$ corresponding to the ideal $\aaa$ and land on the Jacobian variety of a genus-two curve
whose invariants are different to that of $A$. Hence we deduce $\End(A)\subsetneq\O_K$ and therefore $\End(A)=\ZZ[\pi,\overline\pi]$.
The AVIsogenies Magma package \cite{avisogenies} performs this computation in just about three minutes overall on a single Intel i5-8365U CPU core.
\end{example}


\section{Non-simple abelian surfaces}
\label{sec:non-simple}

The endomorphism algebra of non-simple abelian surfaces has dimension greater than $4$.
Since $\QQ(\pi,\bar{\pi})$ has only dimension $1$ or $2$, the strategy consisting in ascending the lattice of orders described previously does not apply, which is why we rely on the more explicit results below.


\subsection{Using coprime isogenies}
\label{sub:useisogeny}

Suppose we want to compute a basis of $\End(A)$. In this section, we prove that this problem reduces to computing the endomorphism rings of two other abelian surfaces $B$ and $C$ connected to $A$ by isogenies of coprime degrees. One could thus compute random isogenies from $A$ with codomain two abelian varieties $B$ and $C$ of which the endomorphism rings are known or simpler to compute.

\begin{proposition}\label{prop:two_paths_for_endring}
Suppose we are given a good representation of isogenies $\varphi : A \rightarrow B$ and $\psi : A \rightarrow C$ of coprime degrees, of their duals, and of a basis of $\End(B)$ and $\End(C)$. Then, one can compute a good representation of a basis of $\End(A)$ in polynomial time in the length of the input.
\end{proposition}
\begin{proof}
 Recall that, for any isogeny $\psi : A \to B$, there is unique isogeny $\hat\psi : B \to A$ such that $\psi \hat\psi=[\deg(\psi))]$.

Let $(\eta_i)_i \subset \End(B)$ and $(\nu_i)_i \subset \End(C)$ be the provided bases.
Let $\beta_i = \widehat\varphi \circ \eta_i \circ \varphi$, and $\gamma_i = \widehat\psi \circ \nu_i \circ \psi$.
The lattices in $\End(A)$ generated by $(\beta_i)_i$ and $(\gamma_i)_i$ are $\Lambda_B = \widehat\varphi \circ \End(B) \circ\varphi$ and $\Lambda_C = \widehat\psi \circ \End(C) \circ\psi$ respectively. We have $\deg(\varphi)^2\End(A) \subset \Lambda_B \subset \End(A)$, and $\deg(\psi)^2\End(A) \subset \Lambda_C \subset \End(A)$. We deduce that $[\End(A) : \Lambda_B]$ and $[\End(A) : \Lambda_C]$ are coprime, since $\deg(\varphi)$ and $\deg(\psi)$ are. This implies $\Lambda_B + \Lambda_C = \End(A)$, hence $\End(A)$ is generated by the union of $(\beta_i)_i$ and $(\gamma_i)_i$. From Lemma~\ref{lemma:scalarproductofisogenies}, we can compute the Gram matrix of this generating set, from which we deduce a basis.
\end{proof}

Sometimes we will not have such coprime degree isogenies at our disposal, for instance in the $p$-rank 0 and $a$-number 1 case.
Nevertheless, in the case of non-simple varieties, we may use the following well-known results.

\begin{lemma}\label{lem:comp}
Let $s_1:A\rightarrow B$ and $s_2:A\rightarrow C$ be isogenies. There exists a third isogeny $s_3:C\rightarrow B$ such that $s_1=s_3s_2$ if and only if $\ker(s_2)\subset\ker(s_1)$.
\end{lemma}

\begin{corollary}\label{cor:comp}
Let $s_1:A\rightarrow B$ and $s_3:C\rightarrow B$ be isogenies. There exists a third isogeny $s_2:A\rightarrow C$ such that $s_1=s_3s_2$ if and only if $\ker(\widehat{s}_3)\subset\ker(\widehat{s}_1)$.
\end{corollary}


\subsection{Structure of the endomorphism algebra}

Recall that an abelian variety $A$ is non simple if and only if it is isogenous to a product of elliptic curves over the algebraic closure.
The non-simplicity of $A$ can be detected by computing the characteristic polynomial $f_A$ of its Frobenius endomorphism \cite[Theorem~6]{howe-zhu} and so the isogeny classes of the elliptic factors follow.

There are three possibilities to consider for the $p$-rank: $0$, $1$, and $2$.
The $p$-rank can be computed as described in Section \ref{sec:simple} via \cite[Theorem~2.9]{maisner-nart}.

\subsubsection{$p$-rank-$0$ case}
\label{sec:non-simple-prank0}
Recall the following result from Oort.

\begin{proposition}[{\cite[Theorem 4.2]{Oort74}}]
For any given three supersingular elliptic curves $E,E_1,E_2$
defined over an algebraically closed field,
there is an isogeny $E^2\sim E_1\times E_2$.
\end{proposition}

Fix $E$ a supersingular elliptic curve.
Following Oort \cite{Oort75}, for every $(i,j)\in\overline{\FF}_p^2$,
we denote by $A_{i,j}$ the abelian surface defined over $\overline{\FF}_p$ through the following diagram: 
$$0\rightarrow \alpha_p\xrightarrow{(i,j)} E\times E\rightarrow A_{i,j}\rightarrow 0.$$


In the $p$-rank-$0$ case, there are two possibilities for the $a$-number: $1$ and $2$.

\begin{proposition}[{\cite[Proposition 11.1]{howe-nart-ritzenthaler}, restating results from \cite{Oort75}}]
Let $A$ be an abelian surface of $p$-rank $0$.
We have $a(A) = 2$ if and only if $A\simeq E\times E$.
We have $a(A) = 1$ if and only if $A\simeq A_{ij}$ for some $[i : j] \in \PP^1(\overline{\FF}_p)\setminus\PP^1(\FF_{p^2})$.
Furthermore, if $a(A) = 1$ then $a(A/\alpha_p) = 2$.
\end{proposition}

As a consequence, in the $p$-rank $0$ case, the endomorphism algebra is isomorphic to $\mathcal{M}_2(\mathcal{B}_{p,\infty})$. In particular, in the $a$-number $2$ case we can solve Problem \ref{prob:abstract} since we have $\End(E\times E)\simeq\mathcal{M}_2(R)$ where $R$ is any maximal order in the quaternion algebra $\mathcal{B}_{p,\infty}$. The computational difficulty here is to find an explicit isomorphism or isogeny from $A$ to $E\times E$. We will address this problem in Section \ref{sec:ellfactor}.

For the remainder of this section, we address the problem of explicitly computing endomorphisms in the $a$-number $1$ case.
Let $A=A_{i,j}\sim (E\times E)/(i,j)(\alpha_p)$ and denote by $\varphi:E\times E\rightarrow A$ the corresponding isogeny.
We have
$$
p\cdot\End(E\times E)\subset\End(A)\subset\frac{1}{p}\End(E\times E)
$$
where $\beta\in\End(E\times E)$ goes to $\frac{1}{p}\varphi\circ\beta\circ\varphi^{-1}$
and $\gamma\in\End(A)$ goes to $\frac{1}{p}\varphi^{-1}\circ\gamma\circ\varphi$.

Given $\begin{pmatrix}x & y \\ z & w\end{pmatrix}\in\End(E\times E)$ we want to determine whether it is of the form
$\frac{1}{p}\varphi^{-1}\circ\gamma\circ\varphi$ for some $\gamma\in\End(A)$.
To this extent, we use the Lemma \ref{lem:comp} together with the following result.

\begin{lemma}
If $x$ is an endomorphism, denote by $\widehat x$ its dual and let $|x|=x\circ\widehat x$ and $\tr(x)=x+\widehat x$.
Given $x$, $y$, $w$ and $z$ four elements of $\End(E)$, we have
\[
\begin{pmatrix}
x & y \\
z & w
\end{pmatrix}
\begin{pmatrix}
\widehat{x}|w|-\widehat{z}w\widehat{y} & \widehat{z}|y|-\widehat{x}y\widehat{w}\\
\widehat{y}|z|-\widehat{w}z\widehat{x}& \widehat{w}|x|-\widehat{y}x\widehat{z}
\end{pmatrix}
=
\big(|x||w|+|y||z|-\tr(x\widehat{z}w\widehat{y})\big)
\begin{pmatrix}
1 & 0\\
0 & 1
\end{pmatrix}.
\]
\end{lemma}

\begin{theorem}
Let $A$ be an abelian surface of $p$-rank $0$ and $a$-number $1$.
Its endomorphism ring is
\[
\End(A)\simeq \left\{
\begin{pmatrix}
x & -\frac{i}{j}x+u\pi \\
z & -\frac{i}{j}z+v\pi
\end{pmatrix}
: x,z,u,v\in\End(E)
\right\}.
\]
where the fraction $\frac{i}{j}$ denotes the corresponding integer modulo $p$.
\end{theorem}

\begin{proof}
Let $m=\begin{pmatrix}x & y \\ z & w\end{pmatrix}\in\End(E\times E)$.
We know that $m\in\End(A)$ if and only if $\ker\varphi\subset\ker m$ and
\[
\ker(\varphi)\subset\ker\begin{pmatrix}\widehat{x}|w|-\widehat{z}w\widehat{y} & \widehat{z}|y|-\widehat{x}y\widehat{w}\\ \widehat{y}|z|-\widehat{w}z\widehat{x}& \widehat{w}|x|-\widehat{y}x\widehat{z}\end{pmatrix}.
\]

Applying Lemma~\ref{lem:comp} to $s_1=m$ and $s_2=\varphi$,
the first condition implies that $y$ and $w$ are of the form
\[
y = -\frac{i}{j}x+u\pi
\qquad \text{and} \qquad
w = -\frac{i}{j}z+v\pi
\]
for some endomorphisms $u$ and $v$.
Then we have $\widehat{y}=-\frac{i}{j}\widehat{x}+V\widehat{u}=-\frac{i}{j}\widehat{x}+u'V$.
Similarly, we have $\widehat{w}=-\frac{i}{j}\widehat{z}+v'V$.
Therefore, the second condition always holds if the first one does.
\end{proof}

\subsubsection{$p$-rank-$1$ case}
\label{sec:non-simple-prank1}
First of all, let us recall that if an abelian surface decomposes as a product of an ordinary and a supersingular elliptic curve, this decomposition occurs on the base field. 

\begin{corollary}[\cite{howe-zhu} and {\cite[Corollary 2.17]{maisner-nart}}]
If an abelian surface $A$ defined over $\FF_q$ decomposes over $\overline{\FF}_q$ as the product of two elliptic curves, one supersingular, the other ordinary, then $A$ decomposes over the base field $\FF_q$.
\end{corollary} 

The following result describes the algebra and the endomorphism ring in the case of $p$-rank $1$.

\begin{proposition}\label{propnsp1}
Let $A/\FF_q$ be a non-simple abelian surface of $p$-rank $1$.
It is isomorphic to a quotient $(E_1\times E_2)/H$ where $E_1$ is an ordinary elliptic curve, $E_2$ is a supersingular one, and $H$ is a finite subgroup. In other words, we have an exact sequence
\[
1\to H\to E_1\times E_2\stackrel{\varphi}{\to} A \to 1
\]
which gives an isomorphism of endomorphism algebras
\[
\QQ\otimes\End(A)\simeq\QQ\otimes\left(\End(E_1)\times\End(E_2)\right).
\]
More precisely, the endomorphism ring $\End(A)$
is the suborder of $\End(E_1)\times\End(E_2)$
made of elements $s$ such that  $H\subset\ker s$.
\end{proposition}

\begin{proof}
The decomposition of the characteristic polynomial of the Frobenius endomorphism implies that we have an isogeny
$A\sim (E_1\times E_2)/H$  where $E_1$ is an ordinary elliptic curve and $E_2$ is a supersingular one.
The rest of the statement follows naturally from the exact sequence.
\end{proof}

As in the previous case, the computational difficulty is finding an isogeny between $A$ and $E_1\times E_2$; see Section \ref{sec:ellfactor}.

\subsubsection{$p$-rank-$2$ case}
\label{sec:non-simple-prank2}
In this case, the surface $A$ decomposes over the algebraic closure as the product of two ordinary elliptic curves. 
\begin{proposition}[\cite{kani}]
Non-simple abelian surfaces $A/\FF_q$ of $p$-rank $2$ are isomorphic to
a product $E_1\times E_2$ of two ordinary elliptic curves and their endomorphism ring is isomorphic to
\[
\begin{pmatrix}
\End(E_1) & \Hom(E_2,E_1) \\
\Hom(E_1,E_2) & \End(E_2)
\end{pmatrix}.
\]
\end{proposition}

When $A$ is the Jacobian variety of a genus-two curve and has $p$-rank $2$,
it is an indecomposable variety and, therefore, admits a principal polarization different from the product one.
Therefore, in this case, we have $\Hom(E_1,E_2)\neq 0$.

\begin{remark}
Since $E_1$ and $E_2$ are ordinary, computing $\Hom(E_1,E_2)$ reduces to the isogeny path problem.
Indeed, if $\End(E_1)=\End(E_2)$, computing $\Hom(E_1,E_2)$ is equivalent to computing the ideal class of the class group corresponding to isogenies from $E_1$ to $E_2$.
If $\End (E_1)\neq\End(E_2)$, one would first compute vertical isogenies from each $E_i$ to an elliptic curve $E_i'$ with endomorphism ring $\End(E_1)+\End(E_2)$ and the problem is then reduced to the previous case.
\end{remark}


\subsection{Computing elliptic factors}\label{sec:ellfactor}
Let $A=\Jac(C)$ be the Jacobian variety of a genus-two curve defined over a finite field $\FF_q$.
In this section, all results are stated over the algebraic closure $\overline\FF_q$ for simplicity; they can be exploited effectively by working on the base field and then extending it as necessary : as already mentioned, all endomorphisms are defined over extensions of the base field of bounded degree \cite[Theorem~2.4]{silverberg-fields-of-def}. 

When $A$ is non-simple, we look for its elliptic factors as elliptic subcovers,
following the work of Kani \cite{kani94,Kani97,kani16,Kani19}.

Suppose that $C$ admits a non-constant morphism $f : C\rightarrow E$ to an elliptic curve $E$.
If $f$ does not factor over an isogeny of $E$, then we say that $f$ is an elliptic subcover of $C$.
Note that this last condition imposes no essential restriction since every nonconstant
$f : C \rightarrow E$ factors over a unique elliptic subcover.

A classical theorem due to Picard \cite{picard} and Bolza \cite{Bolza98} states that a
curve $C$ of genus two has either none, two or infinitely many elliptic subcovers. This
is in part due to the fact that the elliptic subcovers occur in pairs: given an elliptic subcover $f : C\rightarrow E$, there is a canonical “complementary” elliptic
subcover $f': C\rightarrow E'$ of the same degree $ \deg(f) = \deg(f')$ such that the induced maps on the associated Jacobian varieties
fit into an exact sequence
\[0 \rightarrow \Jac(E) \xrightarrow{f_*} \Jac(C)\xrightarrow{f'_*}
\Jac(E')\rightarrow 0.\]

\begin{proposition}\label{existencenn}
Let $A$ be a non-simple Jacobian variety of dimension $2$.
There exists an $(n,n)$-isogeny, preserving the polarization, to a product of elliptic curves with the product polarization.
\end{proposition}

 The case in which $C$ has infinitely many elliptic subcovers happens precisely when the Jacobian of $C$ is $\overline{\FF}_q$-isogenous to $E^2$, for some elliptic curve $E/\overline{\FF}_q$. In particular, in the non-simple $p$-rank 1 case the eliptic curves $E_1$ and $E_2$ and the group $H$ in Proposition \ref{propnsp1} are uniquely determined and the isogeny $A\rightarrow E_1\times E_2$ is an $(n,n)$-isogeny. 

\subsubsection{Bounding the degree and the field of definition of the subcover}
\label{theoretical-bounds}

A bound on the degree of the isogeny may be derived from the following results.

 \begin{theorem}
 Let $(A, \lambda_A)$ be a principally polarized abelian surface,
 denote by $\operatorname{NS}(A, \lambda_A)$ its Néro-Severi group and let $n$ be an integer.
We have a one–to–one correspondence between
\begin{itemize}
\item the set of all elliptic subcovers $E$  of $A$ of degree $n = \deg(E)$, and
\item the set of primitive classes $[D] \in \operatorname{NS}(A, \lambda_A)$ with invariant $\Delta(D) = n^2$
\end{itemize}
given by the embedding $E\to[E] \in\operatorname{NS}(A, \lambda_A)$.
 \end{theorem}

For all $D\in\operatorname{NS}(A, \lambda_A)$, set $q_{(A,\lambda_A)}(D) = (D.\lambda_A)^2 - 2(D.D)$.

\begin{theorem}[{\cite[Theorem~2]{kani16}}]
 The curve $C/\overline{\FF}_q$ has an elliptic subcover of degree $n$ if and only if the
refined Humbert invariant $q_{\Jac(C)}$ primitively represents $n^2$.
 \end{theorem}

 Kani focuses on computing $q_{\Jac(C)}$ given $(E,E', \deg(f))$ and on computing how many curves correspond to a given triple $(E,E',\deg(f))$.
We are instead interested in the opposite direction: given $C$, compute the triples $(E,E',\deg(f))$, the Humbert invariant $q_C$ or simply bounding $\deg(f)$. Unfortunately, we are unaware of an efficient algorithm to compute $q_{\Jac(C)}$. See for example \cite{eda} for a discussion on how such an efficient algorithm could be use to break different isogeny based post-quantum cryptosystems. 

 A bound on the fields of definition follows from a bound on $\deg(f)$.
 
 \begin{proposition}\label{prop:fieldnn} Let $C/\FF_q$ be a genus-two curve and denote by $A=\Jac(C)$ its Jacobian variety.
 Let $\varphi:A\rightarrow B$ be a separable $(n,n)$-isogeny.
 Then $B$ is a principally polarized abelian variety defined over $\FF_{q^r}$ where
 \[
 r=n^3\prod_{\substack{\ell\text{ prime}\\\ell\mid n}}\frac{1}{\ell^3}(\ell+1)(\ell^2+1)
 \]
 and the isogeny $\varphi$ is defined over $\FF_{q^s}$ with $s=wr$ where $w=\#\Aut(C)$.
 \end{proposition}

 \begin{proof} The kernel of a separable $(n,n)$-isogeny is a maximal isotropic subgroup of $A[n]$, and hence $B$ is naturally equipped with a principal polarization \cite[Proposition~2.1]{Dolgachev}. The value of $r$ follows from counting the number of maximal isotropic groups in $A[n]$. The value of $s$ is obtained by counting automorphisms of $A$ preserving the polarization, which is equal to $\#\Aut(C)$ because $C$ is hyperelliptic.
 \end{proof}

 \begin{remark}\label{rem:fieldnn} If $B\sim E_1\times E_2$ with the product polarization then $E_i$ are defined over $\FF_{q^{2r}}$. 
 \end{remark}

We now give two methods to obtain elliptic subcovers of a given Jacobian varietity of a genus-two curve.

\subsubsection{First method: finding $(n,n)$-isogenies to a product of elliptic curves}

According to Proposition \ref{existencenn}, in the non-simple case there are $(n,n)$-isogenies from $A$ to a product of elliptic curves with the product polarization. Since $(n,n)$-isogenies can be computed efficiently \cite{cosset-robert,avisogenies,milio,radical}, we obtain the following algorithm.

\begin{algo}{Finding \unboldmath$(n,n)$-isogenies}
\textsc{Input:} & A curve $C$ of genus two with non-simple Jacobian variety. \\
\textsc{Output:} & An elliptic factor. \smallskip \\
1. & Set $n\gets 2$. \\
2. & Compute the torsion subgroup $\Jac(C)[n]$. \\
3. & For all maximal isotropic subgroups of $\Jac(C)[n]$: \\
4. & \qquad Compute the corresponding $(n,n)$-isogeny $\Jac(C)\to B$. \\
5. & \qquad If $ B$ splits: \\
6. & \qquad \qquad Return its elliptic factor. \\
7. & Set $n\gets n+1$ and go back to Step 2. \\
\end{algo}

For Step 5, see for example \cite{corte-real-costello-frengley}.
This algorithm terminates and a bound on its runtime may be derived from Section~\ref{theoretical-bounds}.

\subsubsection{Second method: computing regular differentials}

Let $C:y^2=F(x)$ be a genus-two curve. Assume that there is an elliptic curve $E:v^2=u^3+au+b$ and a map $C\rightarrow E$ given by $(x,y)\mapsto(u,v)=(f(x,y),g(x,y))=(f_1(x)+yf_2(x),g_1(x)+yg_2(x))$. We necessarily have $g(x,y)^2=f(x,y)^3+af(x,y)+b$
modulo the equation of $C$. This implies
\begin{equation}\label{eq:cover}
\left\{
\begin{array}{@{\displaystyle~}r@{~=~}l}
g_1^2+Fg_2^2 & b+af_1+f_1^3+3f_1f_2^2F,
\smallskip\\
2g_1g_2 & af_2+3f_1^2f_2+f_2^3F.
\end{array}
\right.
\end{equation}
Using the ideas of \cite[Section 6.2]{tony}, we notice that the pushforward of a regular differential of $E$ has to be regular differential of $C$ and hence a linear combination of $\frac{dx}{y}$ and $\frac{xdx}{y}$. This gives:
\begin{equation}\label{eq:cover2}
\left\{
\begin{array}{@{\displaystyle~}r@{~=~}l}
F(2f'_2+f_2) & 2g_1(\alpha x + \beta),
\smallskip\\
2f'_1 & 2g_2(\alpha x + \beta).
\end{array}
\right.
\end{equation}
From which, generically, that is, if $\alpha\neq 0$, we obtain $g_2=\frac{f'_1}{\alpha x + \beta}$ for a choice of linear factor of $f'_1$ and $g_1=\frac{F(2f'_2+f_2)}{2(\alpha x + \beta)}$ for a choice of $f_2$ such that $\alpha x + \beta\mid F(2f'_2+f_2)$.
Together with Equation~(\ref{eq:cover}), this implies $\deg(f_1)=\deg(f_2)+3$, generically.

We thus obtain the following algorithm.

\begin{algo}{Exploiting regular differentials}
\textsc{Input:} & A curve $C: y^2=F(x)$ of genus two with non-simple Jacobian variety. \\
\textsc{Output:} & An elliptic factor. \smallskip \\
1. & Set $d\gets 1$. \\v
2. & Let $w$ and $r$ be as in Proposition~\ref{prop:fieldnn}.\\
2. & For all $f_1\in\FF_{q^{wr}}[x]$ of degree $d+3$: \\
3. & \qquad For all linear factors $t$ of its derivative $f_1'$: \\
4. & \qquad \qquad For all $f_2\in\FF_{q^{2r}}[x]$ of degree $d$  such that $t \mid F(2f_2'+f_2)$: \\
5. & \qquad \qquad \qquad Let $g_1=\frac{F(2f_2'+f_2)}{2t}$ and $g_2=\frac{f_1'}{t}$. \\
6. & \qquad \qquad \qquad If Equation~(\ref{eq:cover}) is satisfied, return the elliptic factor. \\
7. & Set $d\gets d+1$ and go back to Step~2. \\
\end{algo}


\subsection{Supersingular case}
\label{sec:non-simple-ss}

The methods of the previous sections apply also to this case with $n=p$, the characteristic of the base field.
It may however be unpractical to compute $(p,p)$-isogenies.
We now provide an alternative method based on random walk techniques
which may also be adapted to other settings where the isogeny graph has the rapid mixing property.

\begin{proposition}\label{prop:product_supersingular}
There is an algorithm that on input two supersingular elliptic curves $E_1$ and $E_2$ over $\FF_{p^2}$
outputs a basis of $\End(E_1\times E_2)$ in expected time $\sqrt p (\log p)^{O(1)}$. 
\end{proposition}

\begin{proof}
From~\cite[Theorem~8.8]{PW23}, there is an algorithm which finds bases of $\End(E_1)$ and $\End(E_2)$ in time $\tilde O(\sqrt p)$. Within that same running time, one can compute an isogeny $\varphi : E_1 \to E_2$ of degree $2^e$ for some $e \in \NN$ in efficient representation (see, for instance,~\cite[Proposition~8.7]{PW23}; that is simply a baby-step giant-step resolution of the $2$-isogeny path problem). We can ensure that $\varphi$ has cyclic kernel, i.e., it is a non-backtracking path in the $2$-isogeny graph (by greedily pruning backtracking sub-paths).

Then, one can compute in polynomial time a basis of the ideal $I = \Hom(E_2,E_1)\circ\varphi \subset \End(E_1)$ as follows.
This ideal consists in all endomorphisms $\alpha$ such that $\deg \varphi$ divides $\alpha \circ \widehat\varphi$, i.e. such that $(\alpha \circ \widehat\varphi)(E_1[2^e]) = 0$.
Let 
$$I_i = I + 2^i\End(E_1) = \left\{\alpha \in \End(E_1) \mid 2^i \text{ divides } \alpha \circ \widehat\varphi\right\},$$
so that $I_0 = \End(E_1)$ and $I_e = I$. We compute $I_i$ iteratively as follows:
\begin{enumerate}
\item Let $I_0 = \End(E_1)$.
\item For each $0 \leq i < e$, compute $I_{i+1} = \left\{\alpha \in I_i \mid \frac{\alpha \circ \widehat\varphi}{2^i}(E_1[2]) = 0\right\}$, the division $\frac{\alpha \circ \widehat\varphi}{2^i}$ being evaluated  iteratively by~\cite[Theorem~3]{HLMW23}.
\end{enumerate}
We then have a basis of $\Hom(E_2,E_1) = (I \circ \widehat\varphi) / [\deg \varphi]$. Similarly, we have a basis of $\Hom(E_1,E_2)$. We conclude from the fact that
\[\End(E_1\times E_2) = \begin{pmatrix}\End(E_1) & \Hom(E_2,E_1) \\ \Hom(E_1,E_2) & \End(E_2)\end{pmatrix}.\]
\end{proof}

\begin{corollary}
    Assuming that~\cite[Hypothesis~1]{craigsmith} holds, there is an algorithm that on input a superspecial Jacobian $A$ over $\FF_{p^2}$
outputs a basis of $\End(A)$ in expected time $p (\log p)^{O(1)}$.
\end{corollary}

\begin{proof}
    For $\ell\in\{2,3\}$, assuming that~\cite[Hypothesis~1]{craigsmith} holds, 
    there is a bound $n = O(\log(p))$ such that a random path 
    $\varphi_\ell : A \to B_\ell$ of length $n$ in the $(\ell,\ell)$-isogeny 
    graph reaches a target $B$ that is close to uniformly distributed in the set of superspecial abelian surfaces over $\FF_{p^2}$. Then $B_\ell$ is a product $E_1\times E_2$ with probability $\Omega(p^{-1})$. 
    Therefore, one can find such isogenies $\varphi_2$ and $\varphi_3$ in 
    time $p (\log p)^{O(1)}$ from $A$ to products of elliptic curves $B_2$ 
    and $B_3$. One can compute $\End(B_\ell)$ within the claimed running time 
    with  Proposition~\ref{prop:product_supersingular}, and, deduce $\End(A)$ 
    with Proposition~\ref{prop:two_paths_for_endring}
\end{proof}


\section{Surfaces with extra automorphisms}
\label{sec:extra-auto}

When the curve $C$ admits automorphisms other than $\pm\id$,
finding a decomposition of its Jacobian variety $\Jac(C)$ is easier as the following result shows.

\begin{theorem}[{\cite[Theorem B]{kanirosen}}]\label{th:1}
Let $C$ be a curve and let $G$ be a finite subgroup of the automorphism group $\Aut(C)$ such that $G=H_1\cup \cdots \cup H_n$, where the $H_i$'s are subgroups of $G$ with $H_i\cap H_j=\{\id\}$ for $i\neq j$. Then we have an isogeny
$$\Jac(C)^{n-1}\times \Jac(C/G)^g\sim \Jac(C/H_1)^{h_1}\times \cdots \times \Jac(C/H_n)^{h_n}$$
where $g=|G|$, $h_i=|H_i|$ and $C/G$, $C/H_1$,\ldots, $C/H_n$ denote the curves obtained by quotienting $C$ by the subgroups $G,H_1,\ldots, H_n$ respectively.
\end{theorem}

In the specific case of genus-two curves, we have the more explicit statement below.

\begin{theorem}[{\cite[Theorem 2]{Iezzi}}]\label{JacHyper}
 Assume that the polynomial $f(x)$ factors
completely over the finite field $\FF_q$, i.e.
$$f(x)=c\prod_{i=1}^{6}(x-a_i)$$
with $a_i \in \FF_q$ and $a_i \neq a_j$ when $i \neq j$.
Assume that
$$(a_2-a_4)(a_1-a_6)(a_3-a_5)=(a_2-a_6)(a_1-a_5)(a_3-a_4).$$
and set
\begin{align*}
\lambda&=\frac{(a_1-a_3)(a_2-a_4)}{(a_2-a_3)(a_1-a_4)},\qquad
\mu=\frac{(a_1-a_3)(a_2-a_5)}{(a_2-a_3)(a_1-a_5)}, \qquad \text{and} \\
\theta &= c(a_2-a_3)(a_1-a_4)(a_1-a_5)(a_1-a_6).
\end{align*}
Assume moreover that  there exists a square root of $\lambda(\lambda-\mu)$ in the finite field $\FF_q$.
Then the Jacobian of the hyperelliptic curve $C: y^2=f(x)$ decomposes over $\FF_q$ as 
\[ \Jac(C) \sim E_+ \times E_- \]
where $E_+$ and $E_-$ are the elliptic curves defined  by the equations
$$ y^2=\frac{\theta(1-\mu)}{1-\lambda}x(x-1)
\left(x-\frac{(1-\lambda)\left(\mu-2\lambda\pm2\sqrt{\lambda(\lambda-\mu)}\right)}{\mu-1}\right).
$$
\end{theorem}

We now consider each specific case depending on the automorphism group type.
For each such type, Table~\ref{table:extra} gives the associated family of genus-two curves;
see \cite{CardonaQuer} and \cite{Bolza2} for details.

\begin{table}
\begin{center}
\begin{tabular}{c|c}
$\Aut(C)$ & Family \\ \hline
$C_2\times C_5$ & $y^2=x^5-1$\\
$\tilde{S}_4$ & $y^2=x^5-x$\\
$2D_{12}$ & $y^2=x^6-1$\\
$D_{12}$ & $y^2 = x^6 + tx^3 + 1$\\
$D_8$ & $y^2 = x^5 + tx^3 + x$\\
$V_4$ & $y^2 = x^6 + tx^4 + sx^2 + 1$
\end{tabular}
\end{center}
\caption{List of genus-two curves with extra automorphisms.}
\label{table:extra}
\end{table}


\subsection{Automorphism groups admitting a subgroup of type $V_4$}

Except for those with automorphism group $C_2\times C_5$,
all families in the table above are specializations of the family of curves
\[
C_{t,s}:y^2=x^6+tx^4+sx^2+1
\]
over a finite field $\overline{\FF}_p$ with $p$ odd.
The quotient by the automorphism $(x,y)\mapsto(-x,y)$ produces the degreee-two morphism
$$
\phi:\function{C_{t,s}}{E_{t,s}:v^2=u^3+tu^2+su+1}{(x,y)}{(u,v)=(x^2,y).}
$$

The complementary elliptic subcover of degree two is
$$
\phi':\function{C_{t,s}}{E_{s,t}:v^2=u^3+su^2+tu+1}{(x,y)}{(u,v)=(1/x^2,y/x^3),}
$$
which can also be described as the quotient of $C_{t,s}$ by the automorphism $(x,y)\mapsto(-x,-y)$.

These two covers produce a $(2,2)$-isogeny
$$
\Phi=\phi^*\times\phi'^*:
\function{E_{t,s}\times E_{s,t}}{\Jac(C_{t,s})}{(P-\infty,Q-\infty)}{\displaystyle
\sum_{R\in\phi^{-1}(P)}R-\sum_{R\in\phi^{-1}(\infty)}R+\sum_{R\in\phi'^{-1}(Q)}R-\sum_{R\in\phi'^{-1}(\infty)}R,
}
$$
whose kernel is contained in $(E_{t,s}\times E_{s,t})[2]$. Write $u^3+tu^2+su+1=(u-\alpha_1)(u-\alpha_2)(u-\alpha_3)$. Then $u^3+su^2+tu+1=(u-\frac{1}{\alpha_1})(u-\frac{1}{\alpha_2})(u-\frac{1}{\alpha_3})$. Set $P_i^{\pm}=(\pm\sqrt{\alpha_i},0)$, so that $\phi^{-1}((\alpha_i,0))=P_i^+-\infty+P_i^--\infty$. Under the usual identification of $E$ with its Jacobian variety, the kernel of $\Phi$ is
$$
\left\{\infty\times\infty\right\}\cup\left\{(\alpha_i,0)\times\left(\frac{1}{\alpha_i},0\right):i\in\{1,2,3\}\right\}.
$$
Indeed, we have $\Phi\left((\alpha_i,0)\times\left(\frac{1}{\alpha_i},0\right)\right)=\operatorname{div}(x^2-\alpha_i)=0$.

\begin{remark}
The elliptic curves $E_{t,s}$ and $E_{s,t}$ are in general not isogenous to each other.
\end{remark}

We can now use the results of Section~\ref{sub:useisogeny} to determine the endomorphism ring of $\Jac(C)$.
Consider the dual $(2,2)$-isogeny $\widehat\Phi:\Jac(C_{t,s})\to E_{t,s}\times E_{s,t}$.
It allows us to bound the endomorphism ring from above and below as follows
$$
2\End\left(E_{t,s}\times E_{s,t}\right)\subset\End\left(\Jac(C_{t,s})\right)\subset\frac{1}{2}\End\left(E_{t,s}\times E_{s,t}\right)
$$
where the inclusions are given by the maps:
\begin{align*}
&\function{2\End(E_{t,s}\times E_{s,t})}{\Jac\left(C_{t,s}\right)}{2\psi}{\Phi \circ \psi \circ \widehat\Phi}
\\
&\function{\End\left(\Jac(C_{t,s})\right)}{\frac{1}{2}\End\left(E_{t,s}\times E_{s,t}\right)}{\varphi}{\frac{1}{2}\widehat{\Phi}\circ\varphi\circ\Phi}
\end{align*}

In order to finally identify $\End(\Jac(C_{t,s}))$ among orders which satisfy those bounds,
we check which elements $\frac{1}{2}\psi\in\frac{1}{2}\End\left(E_{t,s}\times E_{s,t}\right)$
can be written as $\frac{1}{2}\widehat\Phi\circ\varphi\circ\Phi$, that is, when $\psi=\widehat\Phi\circ\varphi\circ\Phi$.


\subsection{Automorphism type $C_2\times C_5$}

To conclude this section, we focus on the remaining case of the curve $C:y^2=x^5-1$ defined over a finite field $\FF_p$ with $p\neq 2,5$.

We clearly have the inclusion $\mathbb{Z}[\zeta_5]\subset\End(\Jac(C))$.
In the simple case, that is, when $p$ is totally split in $\mathbb{Z}[\zeta_5]$, this is an equality.
In the non-simple case, we can use the results of Section~\ref{sub:useisogeny} since thanks to
\cite[Theorem 1.2]{zaytsev} we do know the $p$-rank and the $a$-number of the Jacobian variety.

\printbibliography

\end{document}